\documentclass[a4paper,11pt]{article}
\usepackage[T1]{fontenc}
\usepackage{amsfonts}
\usepackage{verbatim}
\usepackage{amsmath, amsthm, amssymb, mathrsfs}
\usepackage{graphicx}
\usepackage{textcomp}

\title{An Exact Method For Simulating Rapidly Decreasing Tempered Stable Distributions}
\author{Michael Grabchak\footnote{Email address: mgrabcha@uncc.edu}\\
{\it University of North Carolina Charlotte}}

\begin{document}
\newtheorem{prop}{Proposition}
\newtheorem{thrm}{Theorem}
\newtheorem{defn}{Definition}
\newtheorem{cor}{Corollary}
\newtheorem{lemma}{Lemma}
\newtheorem{remark}{Remark}
\newtheorem{exam}{Example}

\newcommand{\rd}{\mathrm d}
\newcommand{\rE}{\mathrm E}
\newcommand{\tr}{\mathrm{tr}}
\newcommand{\RDTS}{\mathrm{RDTS}}
\newcommand{\TTS}{\mathrm{TTS}}
\newcommand{\iid}{\stackrel{\mathrm{iid}}{\sim}}
\newcommand{\eqd}{\stackrel{d}{=}}
\newcommand{\cond}{\stackrel{d}{\rightarrow}}
\newcommand{\conv}{\stackrel{v}{\rightarrow}}
\newcommand{\conw}{\stackrel{w}{\rightarrow}}
\newcommand{\conp}{\stackrel{p}{\rightarrow}}

\maketitle

\begin{abstract}
Rapidly decreasing tempered stable (RDTS) distributions are useful models for financial applications. However, there has been no exact method for simulation available in the literature. We remedy this by introducing an exact simulation method in the finite variation case. Our methodology works for the wider class of $p$-RDTS distributions.
\end{abstract}

\section{Introduction}

Over the past few years, rapidly decreasing tempered stable (RDTS) distributions  have proven to be useful models for financial applications. They were introduced in \cite{Kim:Rachev:Bianchi:Fabozzi:2010} and have since been used extensively, see e.g.\ the monograph \cite{Rachev:Kim:Bianchi:Fabozzi:2011},  the recent publications \cite{Fallahgoul:Kim:Fabozzi:Park:2019}, \cite{Fallahgoul:Loeper:2019}, \cite{Kim:Jiang:Stoyanov:2019}, and the references therein. These models capture many empirically observed properties of financial returns and, unlike many competitors such as stable and classical tempered stable distributions, all of their exponential moments are finite. This is useful as it allows one to find certain equivalent martingale measures and it is needed to define certain associated GARCH processes, see the discussion in \cite{Bianchi:Rachev:Kim:Fabozzi:2011}. 

Despite the usefulness of RDTS distributions, there are only two general methods for simulation available in the literature. The first is an approximate method based on truncating an infinite series representation, see \cite{Rosinski:Sinclair:2010}. The second is the inverse transform method, which requires the computationally intensive task of numerically inverting the characteristic function. In this paper, we introduce an exact simulation method for the finite variation case. Our approach combines rejection sampling with state-of-the-art methods for simulating from the so-called truncated tempered stable distributions, see \cite{Chi:2012} and \cite{Dassio:Lim:Qu:2020}. Further, our method works for the wider class of $p$-RDTS distributions.

A $p$-RDTS distribution is an infinitely divisible distribution with no Gaussian part and a L\'evy measure of the form
\begin{eqnarray}\label{eq: gen levy meas}
D x^{-1-\alpha} e^{-a^p x^p}1_{[x<0]}\rd x + C x^{-1-\alpha} e^{-b^p x^p}1_{[x>0]}\rd x
\end{eqnarray}
where $\alpha<2$, $p>1$, $C,D\ge0$, and $a,b>0$. When $p=2$, these correspond to the class of RDTS distributions and in the limiting case when $p=1$ we get the class of classical tempered stable distributions. Other closely related distributions are studied in \cite{Rosinski:2007}, \cite{Rosinski:Sinclair:2010}, \cite{Bianchi:Rachev:Kim:Fabozzi:2011}, \cite{Grabchak:2012} and \cite{Grabchak:2016}. The following result follow immediately from Theorems 21.9 and 25.17 in \cite{Sato:1999}.

\begin{prop}
1. If $X$ has a $p$-RDTS distribution (with $p>1$), then
\begin{eqnarray}\label{eq: mgf exists}
\rE\left[e^{\theta X}\right]<\infty \ \ \mbox{for every } \theta\in\mathbb R.
\end{eqnarray}
2. A $p$-RDTS distribution corresponds to a L\'evy process with finite variation if and only if $\alpha<1$.
\end{prop}

In the limiting case when $p=1$, exact simulation techniques are well-known for the case where $\alpha<1$, see e.g.\ \cite{Hofert:2011} and the references therein. When $p>1$, the only case where an exact simulation method is known is when $p=2$, $\alpha=1$, and the distribution is symmetric, i.e.\ $C=D$ and $a=b$, see Example 5.2 in \cite{Dassio:Lim:Qu:2020}. However, this method does not seem to be generalizable to other $p$-RDTS distributions. In the next two sections, we discuss our simulation methods. Throughout, we write $\mathrm{Pois}(C)$ to denote the Poisson distribution with mean $C$, $U(0,1)$ to denote the uniform distribution on $(0,1)$, and $\mathrm{beta}(\beta,\eta)$ to denote the beta distribution with shape parameters $\beta>0$ and $\eta>0$.

\section{Simulation when $\alpha\in[0,1)$}

Fix $\alpha\in[0,1)$ and $p>1$. When $D=0$, the characteristic function of a $p$-RDTS distribution $\mu$ is given, for $z\in\mathbb R$, by 
\begin{eqnarray}\label{eq: char func}
\hat\mu(z) &=& \exp\left\{C\int_0^\infty \left(e^{ixz}-1\right)x^{-1-\alpha}e^{-b^p x^p}\rd x\right\}\\
&=& \exp\left\{cb^\alpha\int_0^\infty \left(e^{ixz/b}-1\right)x^{-1-\alpha}e^{-x^p}\rd x\right\}.\nonumber
\end{eqnarray}
We denote this distribution by $\RDTS_\alpha^p(b,C)$. It follows that, if $X_1\sim \RDTS_\alpha^p(1,C)$ and $X_2\sim \RDTS_\alpha^p(1,D)$ are independent, then the distribution of $X_1/b-X_2/a$ has the L\'evy measure given in \eqref{eq: gen levy meas}. Thus, without loss of generality, we focus on simulation from $\RDTS_\alpha^p(1,C)$.

Our simulation method requires simulating from a truncated tempered stable distribution. This is an infinitely divisible distribution whose characteristic function is given by
\begin{eqnarray*}
\exp\left\{C\int_0^1 \left(e^{ixz}-1\right)x^{-1-\alpha}e^{- x}\rd x\right\}
\end{eqnarray*}
for some $C>0$ and $\alpha\in[0,1)$. We denote this distribution by $\TTS_\alpha(C)$. For $\alpha\ne0$, two exact simulation techniques are available in the literature. The first is given in Algorithm 4.4 of \cite{Dassio:Lim:Qu:2020} and the second requires combining Algorithm 5.1 in \cite{Chi:2012} with the algorithm in Section 2.2 of that paper. When $\alpha=0$ the $\TTS_0(C)$ distribution corresponds to the truncated gamma distribution, which is closely related to the Dickman distribution and Vervaat perpetuities. Exact simulation approaches can be found in Algorithm 3.2 of \cite{Dassio:Qu:Lim:2019} or Steps 1-4 in Algorithm 6.1 of \cite{Chi:2012}, see also the references in these papers for additional simulation techniques. Our methodology is based on the following result. 

\begin{thrm}\label{thrm:main}
Fix $\alpha\in[0,1)$, $p>1$, and $C>0$. If $X\sim\RDTS_\alpha^p(1,C)$ then 
\begin{eqnarray}\label{eq: main}
X \eqd X_0 + \sum_{j=1}^{N_1} X_j + \sum_{j=1}^{N_2} Y_j,
\end{eqnarray}
where $N_1,N_2,X_0,X_1,X_2,\dots$, and $Y_1,Y_2,\dots$ are independent random variables with\\
1. $N_1\sim\mathrm{Pois}(CK_1)$ and $N_2\sim\mathrm{Pois}(CK_2)$, where
$$
K_1 = \int_1^\infty e^{-x^p}x^{-1-\alpha}\rd x= \frac{1}{p} \int_1^\infty e^{-x}x^{-1-\alpha/p}\rd x
$$
and
$$
K_2 = \int_0^1  \left(e^{- x^p}-e^{-x}\right)x^{-\alpha-1}\rd x;
$$
2. $X_0\sim \TTS_{\alpha}(C)$;\\
3. $X_1,X_2,\dots$ all have pdf
$$
f_1(x)=\frac{1}{K_1} e^{-x^p}x^{-1-\alpha}, \ \ x>1;
$$
4. $Y_1,Y_2,\dots$ all have pdf
$$
f_2(x)=\frac{1}{K_2} \left(e^{- x^p}-e^{-x}\right)x^{-\alpha-1}, \ \ 0<x<1.
$$
\end{thrm}

\begin{proof}
Note that
$$
e^{-x^p} 1_{[x>0]}=  e^{-x} 1_{[0<x\le1]}+e^{-x^p} 1_{[x>1]}+ \left(e^{- x^p}-e^{-x}\right) 1_{[0<x\le1]}.
$$
It follows that the characteristic function of $\RDTS_\alpha^p(1,C)$ can be written as
\begin{eqnarray*}
&&\exp\left\{C\int_0^1 \left(e^{ixz}-1\right)x^{-1-\alpha}e^{- x}\rd x\right\}\\
&&\qquad\times\exp\left\{C\int_1^\infty \left(e^{ixz}-1\right)x^{-1-\alpha}e^{- x^p}\rd x\right\}\\
&&\qquad\times\exp\left\{C\int_0^1 \left(e^{ixz}-1\right)x^{-1-\alpha}\left(e^{-x^p}-e^{-x}\right)\rd x\right\}.
\end{eqnarray*}
In this product, the first term is the characteristic function of $\TTS_{\alpha}(C)$ and, by a simple conditioning argument, it is readily seen that the other two terms are the characteristic functions of the sums in \eqref{eq: main}.
\end{proof}

Applying change of variables and integration by parts shows that we can write 
$$
K_1 = \alpha^{-1}\left(\Gamma(1-\alpha/p,1)+e^{-1}\right)
$$
and
$$
K_2 = \alpha^{-1}\left(\gamma(1-\alpha,1)-\gamma(1-\alpha/p,1)\right),
$$
where $\gamma(\beta,t)=\int_0^t e^{-x} x^{\beta-1}\rd x$ and $\Gamma(\beta,t)=\int_t^\infty e^{-x} x^{\beta-1}\rd x$ are, respectively, the lower and the upper incomplete gamma functions.

To use Theorem \ref{thrm:main} for simulation, we just need a way to simulate from the pdfs $f_1$ and $f_2$. Our approach is based on rejection sampling. First, note that 
$$
f_1(x)=\frac{1}{K_1} e^{-x^p}x^{-1-\alpha}1_{[x>1]}\le \frac{1}{epK_1} pe^{1-x^p}x^{p-1}1_{[x>1]}= V_1g_1(x),
$$
where $V_1 = 1/(epK_1)$ and $g_1(x)=pe^{1-x^p}x^{p-1}1_{[x>1]}$. It is easy to check that, if $U\sim U(0,1)$ and $Y=(1-\log U)^{1/p}$, then $Y$ has pdf $g_1$. This leads to the following algorithm for simulating from $f_1$.\\

\noindent \textbf{Algorithm 1.}\\
\textbf{Step 1.} Independently simulate $U_1,U_2\sim U(0,1)$ and let $Y=(1-\log(U_2))^{1/p}$.\\
\textbf{Step 2.} If $U_1\le Y^{-p-\alpha}$ return $Y$, otherwise go back to step 1.\\

From standard facts about rejection sampling, on a given iteration the probability of accepting is
$$
1/V_1 = epK_1 \ge e \int_1^\infty e^{-x}x^{-2}\rd x \approx 0.40365.
$$
Thus, this probability is uniformly bounded away from zero in all parameters.

Next, we turn to simulation from $f_2$. We begin by recalling the fact that for $x\in[0,1]$ and $p>1$
\begin{eqnarray}\label{eq: exp bound}
0\le e^{-x}(x-x^p) \le e^{-x^p}-e^{-x}\le x-x^p,
\end{eqnarray}
see e.g.\ Lemma 7.2 of \cite{Grabchak:2020}. It follows that
\begin{eqnarray*}
f_2(x) \le \frac{1}{K_2} x^{-\alpha}(1-x^{p-1})1_{[0<x\le1]} = V_2 g_2(x),
\end{eqnarray*}
where $V_2 = \frac{p-1}{(p-\alpha)(1-\alpha)K_2}$ and 
$$
g_2(x) = \frac{(p-\alpha)(1-\alpha)}{p-1}x^{-\alpha}(1-x^{p-1}), \ \ 0<x\le1.
$$
It is easy to check that, if $Y\sim \mathrm{beta}\left(\frac{1-\alpha}{p-1},2\right)$, then $Z=Y^{1/(p-1)}$ has pdf $g_2$. Letting
$$
\varphi(z) = \frac{e^{-z^p}-e^{-z}}{z-z^p}
$$
leads to the following algorithm for simulating from $f_2$.\\

\noindent \textbf{Algorithm 2.}\\
\textbf{Step 1.} Independently simulate $U\sim U(0,1)$, $Y\sim \mathrm{beta}\left(\frac{1-\alpha}{p-1},2\right)$, and let $Z=Y^{1/(p-1)}$.\\
\textbf{Step 2.} If $U\le \varphi(Z)$ return $Z$, otherwise go back to step 1.\\

In this case the probability of accepting on a given iteration is
$$
1/V_2 = K_2 \frac{(p-\alpha)(1-\alpha)}{p-1}\ge e^{-1}\approx 0.36788
$$
where we use \eqref{eq: exp bound} to get
$$
K_2 \ge \int_0^1  e^{- x}x^{-\alpha}(1-x^{p-1})\rd x \ge e^{-1} \int_0^1 x^{-\alpha}(1-x^{p-1})\rd x=e^{-1} \frac{p-1}{(p-\alpha)(1-\alpha)}.
$$

\section{Simulation when $\alpha<0$}

Simulation when $\alpha<0$ is significantly simpler. However, this case still leads to interesting models, which have found application in finance \cite{Carr:Geman:Madan:Yor:2002} and other areas \cite{Aalen:1992}. In this case, we allow for any $p>0$ and not just $p>1$. However, in the case when $p\in(0,1]$, the result in \eqref{eq: mgf exists} will not hold.

Fix $\alpha<0$, $p>0$, $b>0$, $C>0$, and, as before, without loss of generality assume that $D=0$. In this case the characteristic function is again given by \eqref{eq: char func} and we again denote the corresponding distribution by $\RDTS^p_\alpha(b,C)$. Let
$$
K_3 = \int_0^\infty e^{-bx^p}x^{-1-\alpha}\rd x = b^\alpha p^{-1}\Gamma(1-\alpha/p)<\infty.
$$
Applying a simple conditioning argument shows that if $X\sim \RDTS^p_\alpha(b,C)$, then 
$$
X \eqd \sum_{j=1}^N X_j
$$
where $N,X_1,X_2,\dots$ are independent random variables with $N\sim\mathrm{Pois}(CK_3)$ and the $X_i$'s all having pdf
$$
f_3(x) = \frac{1}{K_3} e^{-bx^p}x^{|\alpha|-1}, \ \ x>0.
$$
This is a generalized gamma distribution, which was introduced in \cite{Stacy:1962}, and we denote it by $\mathrm{GGa}(|\alpha|,p,b)$. When $p=1$ it reduces to the standard gamma distribution, which we denote by $\mathrm{Gamma}(b,|\alpha|)$. It is easily checked that
\begin{eqnarray*}
\mbox{if }X\sim \mathrm{Gamma}(|\alpha|/p,b)\mbox{, then }X^{1/p}\sim \mathrm{GGa}(|\alpha|,p,b),
\end{eqnarray*}
which gives a simple method for simulation.

\section{Discussion}

The difficulty in working with $p$-RDTS distributions is that they do not have a closed form for their pdfs or cdfs. In fact, aside for the notable case when $p=2$, they do not even have a closed form for their characteristic functions. Some of this difficulty can be alleviated by having the kind of tractable simulation method introduced here. For instance, one can perform parameter estimation by using a simulation-based approach such as the simulated quantile method of \cite{Dominicy:Veredas:2013}. We will explore parameter estimation for $p$-RDTS distributions using such approaches in a future work. In a different direction, it may be possible to develop simulation methods similar to the one described in this paper for other, related, distributions such as certain classes of the generalized tempered stable distributions of \cite{Rosinski:Sinclair:2010} or the $p$-tempered $\alpha$-stable distributions of \cite{Grabchak:2012} and \cite{Grabchak:2016}.

\end{document}